\documentclass[12pt]{article}
\usepackage{amsfonts, amsthm, amsmath, amssymb}

\theoremstyle{plain}
\newtheorem{theorem}{Theorem}[section]
\newtheorem{lemma}[theorem]{Lemma}

\newtheorem{corollary}[theorem]{Corollary}
\newtheorem*{theorem*}{Theorem}

\theoremstyle{remark}
\newtheorem{remark}[theorem]{Remark}

\theoremstyle{definition}
\newtheorem{definition}[theorem]{Definition}

\numberwithin{equation}{section}

\begin{document}

\title{Convergence of equilibrium measures corresponding to finite
subgraphs of infinite graphs: new examples}

\author{B.M. Gurevich \\
Mech. and Math. Dept. of Moscow State University,\\
Institute for Information Transmission Problems\\
of the Russian Academy of Sciences}



\date{}

\maketitle

\bigskip

\begin{abstract}
A problem from thermodynamic formalism for countable symbolic Markov
chains is considered. It concerns asymptotic behavior of the
equilibrium measures corresponding to increasing sequences of finite
sub-matrices of an infinite nonnegative matrix  $A$ when these
sequences converge to $A$. After reviewing the results obtained up
to now, a solution of the problem is given for a new matrix class. A
geometric language of loaded graphs instead of the matrix language
is used.

\end{abstract}

\section{Introduction}

\label{introd}
\smallskip

This paper is about asymptotic behavior of the equilibrium measures
corresponding to increasing sequences of finite sub-matrices of an
infinite nonnegative matrix. This is a part of thermodynamic
formalism for the countable symbolic Markov chains.

A starting material for developing thermodynamic formalism is a pair
$(T,f)$, where $T$ is a map of some topological space $X$ to itself
and $f$ a continuous function from $X$ to $\mathbb R$. In the case
under consideration, $T$ is a Markov shift in the space $X$ of the
sequences $x=(x_i,-\infty<i<\infty)$, $x_i\in V$, where $V$ is an
infinite countable set and $f$ depends on a finite number of
coordinates $x_i$ (without loss of generality only dependence on
$x_0$ and $x_1$ can be assumed).

M.I. Malkin \cite{M1} was the first who considered asymptotic
behavior of the equilibrium measures corresponding to an increasing
sequence of finite sub-matrices $A_n$ of an infinite nonnegative
matrix $A$ when $A_n\to A$ in a natural sense. A more general
approach was suggested in \cite{GS}. Already in that work a
geometric language based on the notion of a loaded graph and
equivalent to the matrix language was used together with the latter.
In some cases the geometric language is more relevant. In the
present paper it is used as well.

The answer to the question we address depends on the type of the
loaded graph (or the nonnegative matrix) under consideration. We
mean the classification of the loaded graphs which is almost
completely similar to that of the countable state Markov chains: we
have transient and recurrent loaded graphs and, among the latter,
null-recurrent and positive recurrent loaded graphs; finally,
positive recurrent loaded graphs can be stable positive and unstable
positive.

The last class differs from the others in that it demonstrates some
diversity in behavior of subgraphs: an infinite graph can have two
sequences of finite subgraphs one of which induces the equilibrium
measures converging to the equilibrium measure of the whole graph,
while the equilibrium measures of the other sequence tend to the
zero measure. It was conjectured that this is the case for every
unstable positive loaded graph, but now we know only some families
of such graphs. In what follows we describe these families and
present a new one essentially different from them.

\section{Preliminary information and statement of the problem}
\label{Notat}
\smallskip

This section contains some definitions and facts from thermodynamic
formalism for countable state Markov chains, primarily those used
below; see \cite{GS} and \cite{GN} for more detailed presentation.

Let $G=(V,E)$ be a directed graph with finite or infinite countable
vertex set $V$ and edge set $E\subset V\times V$; we write $V=V(G)$,
$E=E(G)$. If $e=(u,v)\in E$, we write $s(e)=u$, $t(e)=v$ and refer
to $e$ as the edge going from $u$ to $v$ or the edge with initial
vertex $u$ and terminal vertex $v$. A path $\gamma$ of length $\ell$
in $G$ going from $u$ to $v$ is a sequence of edges
$\gamma=e_1,\dots,e_{\ell}$ such that
$$
s(e_1)=u,\ \ t(e_{\ell})=v,\ \ s(e_{i+1})=t(e_i),\ \ 1\le
i\le\ell-1.
$$
We write $e'\in\gamma$ if $e'=e_i$ for some $i\in\{1,\dots,\ell\}$,
and we write $v'\in\gamma$ if $v'=s(e_i)$ or $v'=t(e_i)$ for some
$i$.

\begin{remark}
\label{remark1} It is sometimes convenient to treat a path as a
sequence of vertices $v_1,\dots,v_n$ such that $(v_i,v_{i+1})\in E$,
$i,\dots,n-1$, rather than a sequence of edges. The match between
this definition and the previous one is evident and will be used
below by default.
\end{remark}

We say that the above path {\it goes through} the vertices $s(e_i)$,
$2\le i\le\ell$, and $t(e_j)$, $1\le j\le\ell-1$. There can be paths
that go through their initial and/or terminal vertices, i.e., visit
these vertices at least twice. A graph is referred to as {\it
connected} if for each pair $u,v$ of its vertices there is a path
going from $u$ to $v$. In what follows we always assume that the
infinite graph $G$ under consideration is connected.

Let $G=(V,E)$ and $G\,'=(V',E\,')$ be two graphs. One says that
$G\,'$ is a subgraph of $G$ (and writes $G\,'\subset G$) if
$V'\subset V$ and $E\,'\subset E$.

Let $\mathcal W:E(G)\to(0,\infty)$ be a function on $E(G)$ (a {\it
weight function}). The pair $(G,\mathcal W)$ is referred to as a
{\it loaded graph} (LG for short). The number $\mathcal W(e)$, $e\in
E$, is the weight of $e$, and the product
$$
\mathcal W(\gamma):=\prod_{i=1}^\ell\mathcal W(e_i)
$$
is the {\it weight of the path} $\gamma=e_1,\dots,e_{\ell}$. For
every vertex $v\in V(G)$ we denote by $\Gamma_{v,v}$ the family of
all paths $\gamma$ in $G$ going from $v$ to $v$, and by
$\Gamma_{v,v|v}$ the set of all $\gamma\in\Gamma_{v,v}$ that do not
pass through $v$. The elements of $\Gamma_{v,v}$ and
$\Gamma_{v,v|v}$ will be referred to as $v$-{\it cycles} and simple
$v$-{\it cycles}, respectively.

For a fixed $v\in V(G)$, let $q_v(n)$ be the sum of weights over all
simple $v$-cycles of length $n$ and
\begin{equation}
\label{GenFunction} \Phi_{v,v|v}(G,\mathcal W,z):=\sum_{n=1}^\infty
q_v(n)z^n=\sum_{\gamma\in\Gamma_{v,v|v}}\mathcal
W(\gamma)z^{\ell(\gamma)}.
\end{equation}
The radius of convergence of this power series is denoted by
$R_v=R(G,\mathcal W,v)$. Assume that $R_v>0$. Then $R_{v'}>0$ for
all $v'\in V(G)$ (since $G$ is connected). In what follows we use
the function $\Phi_{v,v|v}(G,\mathcal W,z)$ only for $z\ge0$.

\begin{definition}
\label{exhaust} A sequence of subgraphs $G_k=(V_k,E_k)$ of an
infinite graph $G=(V,E)$ is said to be {\it exhaustive} if
$$
G_k\subset G_{k+1},\ k\ge1,\ \ \bigcup_{k\ge 1}V_k=V,\ \
\bigcup_{k\ge 1}E_k=E.
$$
\end{definition}

\begin{definition}
\label{positive} An LG $(G,\mathcal W)$ is said to be {\it stable
positive} (for short SPLG), {\it unstable positive} (UPLG), and {\it
null-recurrent} (NRLG) if there is a vertex $v\in V(G)$ such that
\begin{equation}
\label{stable} \Phi_{v,v|v}(G,\mathcal W,R_v)>1,
\end{equation}
\begin{equation}
\label{unstable} \Phi_{v,v|v}(G,\mathcal W,R_v)=1, \ \
\frac{d}{dz}\Phi_{v,v|v}(G,\mathcal W,z)|_{z=R_v}
 <\infty,
\end{equation}
and
\begin{equation}
\label{0-recurrent} \Phi_{v,v|v}(G,\mathcal W,R_v)=1, \ \
\frac{d}{dz}\Phi_{v,v|v}(G,\mathcal W,z)|_{z=R_v}=\infty,
\end{equation}
respectively. SPLGs and UPLGs together form the set of {\it positive
recurrent} LGs (PRLGs).
\end{definition}

Since $G$ is connected, the conditions \eqref{stable} --
\eqref{0-recurrent} are obeyed by all $v\in V(G)$ if so is by at
least one.

In \cite{GS}, \S3 one can find some characteristic properties of
SPLGs and UPLGs as well as justification of the terminology used
here.

PRLGs and NRLGs together form the set of {\it recurrent} LGs. UPLGs
are intermediate in this set between SPLGs and NRLGs. This fact is
most conspicuous in behavior of exhaustive sequences of finite
subgraphs: while in every of the two extreme classes all such
sequences behave in some sense similarly (although differently in
different classes), for each UPLG studied so far there are two kinds
of exhaustive sequences: some of the sequences behave as all
exhaustive sequences in SPLGs, others do as all exhaustive sequences
in NRLGs. To be more specific we have to remind several more facts
from thermodynamic formalism for countable symbolic Markov chains
(see \cite{GS}).

Let $\Omega$ denote the set of the two-sided infinite paths in $G$,
i.e. (see remark \ref{remark1}),
$$
\Omega(G)=\{\omega=(\omega_i)|_{i\in\mathbb
Z}:(\omega_i,\omega_{i+1})\in E \ \text{for all}\ i\in\mathbb Z\}.
$$
If $(G,\mathcal W)$ is a PRLG, then there is a measure
$\mu^{G,\mathcal W}$ that maximizes some functional on $\mathcal
M(G)$, the set of the shift invariant probability measures on
$\Omega(G)$, and is referred to as an {\it equilibrium} (more
precisely, $\mathcal W$-equilibrium) measure. It is uniquely
determined by the above conditions and is a Markov measure. Outside
the class of PRLGs such a measure does not exist. The explicit forms
of both the measure $\mu^{G,\mathcal W}$ and the functional that
determines it are inessential for what follows (this can be found in
\cite{GS}, \S\,4).

If $G\,'$ is a finite connected graph, then $(G\,',W\,')$ is a SPLG
for every weight function $W\,'$, and hence the measure
$\mu^{G\,',W\,'}$ exists. This holds true in particular for each
finite connected $G\,'\subset G$. Since the path space
$\Omega(G\,')$ is contained in $\Omega(G)$, the measure
$\mu^{G\,',W\,'}$ can be taken as defined on $\Omega(G)$.

One deals with this situation in case of exhaustive sequences of
finite subgraphs of an infinite graph $G$. Let $\{G_k,\,k\in\mathbb
Z\}$ be such a sequence and $\mathcal W$ a weight function on
$E(G)$. Then one can consider the sequence of loaded subgraphs
$(G_k,\mathcal W_k)$, where $\mathcal W_k$ is the restriction of
$\mathcal W$ to $E(G_k)$; we will write $(G_k,\mathcal W)$ instead
of $(G_k,\mathcal W_k)$. A challenge is to clarify behavior of the
sequence of measures $\mu^{G_k,\mathcal W}$ as $k\to\infty$. It is
known (see \cite{GS}, \S\,6) that if $(G,\mathcal W)$ is stable
positive, then $\mu^{G_k,\mathcal W}$ converges weakly to
$\mu^{G,\mathcal W}$, and if $(G,\mathcal W)$ is null-recurrent,
then $\mu^{G_k,\mathcal W}$ converges to the zero measure, while if
$(G,\mathcal W)$ is unstable positive, then it is possible both. We
say that a sequence $\{G_k\}$ is {\it regular} or {\it irregular}
relative to $\mathcal W$ if the former or latter situation is
realized, respectively.

It has been repeatedly conjectured that if $G$ is a countable graph
and $\mathcal W$ a weight function such that $(G,\mathcal W)$ is an
UPLG, then both regular and irregular (relative to $\mathcal W$)
sequences exist in $G$. Since this conjecture is not proved, it is
desirable to extend the family of graphs $G$ for which it can be
confirmed as much as possible.

Two classes of such graphs have been known up to now. One of them is
formed by the so called "petal"\ graphs (see \cite{GS}, Def. 5.14),
in which all cycles go through a unique vertex $v$ and every two
simple $v$-cycles have no common vertex other than $v$. A wider
class is formed by "bundles of cascade graphs" \cite{GN}, where all
cycles go through a fixed finite vertex set $V_0$ and for each
vertex $v$, the common vertices for every pair of simple $v$-cycles
are contained in $V_0$. This property is essentially used in proving
the above-stated conjecture. In the next section we introduce a
class of "linear graphs", which do not have this property, while for
them, the above conjecture is true as before. In the subsequent two
sections this will be proved by means of some new arguments.

We finish this section with the following criterium, which
immediately follows from Theorem 6.4 in \cite{GS}.
\begin{lemma}
\label{critReg} Let $(G,\mathcal W)$ be a UPLG and subgraphs $G_k$
form an exhaustive sequence. Let $R(k)>0$ be such that
$\Phi_{v,v|v}(G_k,\mathcal W,R(k))=1$. Then $\{G_k\}$ is regular or
irregular if and only if there exists a vertex $v\in V$ such that
\begin{equation}
\label{regular}
\lim_{k\to\infty}\frac{d}{dz}\Phi_{v,v|v}(G_k,\mathcal
W,z)|_{z=R(k)}=\frac{d}{dz}\Phi_{v,v|v}(G,\mathcal W,z)|_{z=R_v},
\end{equation}
or the limit equals $\infty$, respectively.
\end{lemma}

\section{Linear graphs}
\label{LinGraph}

In this section we use the notation introduced in the previous one.

Consider a directed graph $G=(V,E)$, where the vertex set $V$ is
$\mathbb N$, the set of positive integers, and the edge set $E$
falls into two subsets, $E_1$ and $E_2$, where $E_1$ consists of all
pairs of the form $(i,1)$ and $(i,i+1$), $i\in\mathbb N$, while
$E_2$ consists of some pairs of the form $(i,j)$ such that $i<j-1$
and the following {\it finiteness condition} holds: for each $i$,
there is only a finite number of vertices $j$ such that $(i,j)\in
E_2$. Such a graph $G$ will be referred to as a {\it linear graph}.

Notice that the vertex $1\in V(G)$ belongs to every cycle in $G$,
and this resembles a petal graph (see \S\,\ref{Notat}). For
convenience of reference to \S\,\ref{Notat} we sometimes denote this
vertex by $v$. The edges of the form $(i,1)$, and $(i,j)$ with
$j>i+1$ will be called {\it backward} and {\it forward} jumps,
respectively. It is clear from the structure of $G$ that every
simple $v$-cycle contains a unique backward jump. For a $v$-cycle
$\gamma$ let $M(\gamma)$ denote its {\it maximal vertex}, i.e.,
$M(\gamma):=\max\{i\in V:i\in\gamma\}$.

For every $n\ge 1$ we set $k_n=:\max_\gamma M(\gamma)$, where max is
over all simple $v$-cycles of length $n$.

We start the proof of the conjecture for a linear graph with
constructing an irregular sequence of its finite subgraphs.

\section{Existence of an irregular sequence}
\label{Irregular}
\smallskip

Let $\Gamma$ be a family of paths in the graph $G$ and $G(\Gamma)$
the minimal subgraph of $G$ containing all paths from $\Gamma$. We
say that $G(\Gamma)$ is generated by $\Gamma$. For a linear graph
$G$ introduced in Section \ref{LinGraph}, denote by $\Gamma_\ell$
the family of all simple $v$-cycles of length $\ell$ in $G$,
$\ell\ge1$, and set
\begin{equation}
\label{sub-graphs} G_n:=G(\cup_{l\le n}\Gamma_\ell),\ \
G_{n,m}:=G((\cup_{\ell\le n}\Gamma_\ell)\cup\Gamma_m),\ m>n.
\end{equation}
Clearly $G_n$ and $G_{n,m}$ are finite graphs. Since they are
generated by families of $v$-cycles, they are connected and the sets
of their vertices and edges coincide with the corresponding sets in
the generating families. However they can contain $v$-cycles that
are absent in these families, because some $v$-cycles in
$\Gamma_{n,m}$ have common vertices different from $v$ and common
edges. (This situation severely hampers constructing regular and
irregular sequences of finite subgraphs.)

Consider some other properties of the graphs $G_n$ and $G_{n,m}$. To
begin with we obtain from the sequence $\{k_n\}$ defined in
\S\,\ref{LinGraph} another sequence $\{a(n)\}$ as follows: $a(n)=1$
for $1\le n\le k_2$ and $a(n)=i$ as $k_i<n\le k_{i+1}$, $i\ge2$. It
is evident that
\begin{equation}
\label{k(a(n))} k_{a(n)}<n \ \ {\rm as}\ \ n\ge k_1+1.
\end{equation}
\begin{lemma} \label{G(mn)} 1. The length of each simple $v$-cycle in $G_n$
does not exceed $k_n$.
\item{2}. There are no simple $v$-cycles in $G_{n,m}$ such that
\begin{equation}
\label{<l<} k_n<\ell(\gamma)<a(m).
\end{equation}
\end{lemma}
\begin{proof} 1. Clearly every simple $v$-cycle in $G$ is a path from $v$ to
$M(\gamma)$ extended by the backward jump $(M(\gamma),v)$. But there
are no vertices $i>k_n$ in $G_n$, which implies the first statement
of the lemma.

2. The second statement is worthy of proving only when $k_n<a(m)$.
Note that there are no backward jumps $(i,v)$ with $i>k_n$ and with
$i<m$ in $v$-cycles from $\cup_{\ell\le n}\Gamma_\ell$ and from
$\Gamma_m$, respectively. Hence there are no edges $(i,v)$ with
$k_n<i<m$ in $G_{n,m}$.

We now assume that there is a simple $v$-cycle $\gamma$ in $G_{n,m}$
such that \eqref{<l<} asserts. Then either $M(\gamma)\le k_n$, or
$M(\gamma)\ge m$. In the former case $\ell(\gamma)\le k_n$, while if
the latter case takes place and $\ell(\gamma)<a(m)$, then
$M(\gamma)\le k_{a(m)}<m$ (see \eqref{k(a(n))}). But we saw above
that $M(\gamma)\ge m$ for all $v$-cycles $\gamma$ with
$M(\gamma)>k_n$ in $G_{m,n}$. This contradiction completes the
proof.
\end{proof}

From now on we assume that $\mathcal W$ is a weight function such
that $(G,\mathcal W)$ is a UPLG.

\begin{corollary}
\label{GenFun} For all $n,m\in\mathbb N$ such that $m>n$, the
functions $\varphi_n(z):=\Phi_{v,v|v}(G_n,\mathcal W,z)$ and
$\varphi_{n,m}(z):=\Phi_{v,v|v}(G_{n,m},\mathcal W,z)$ are of the
form
\begin{equation}
\label{genFunct1} \varphi_n(z)=\sum_{i=1}^nq(i)z^i+
\sum_{i=n+1}^{k_n}q^{(1)}(n,i)z^i,
\end{equation}
\begin{equation}
\label{genFunct2}
\varphi_{n,m}(z)=\varphi_n(z)+\sum_{i=a(m)}^{k_m}q^{(2)}(m,i)z^i,
\end{equation}
where
\begin{equation}
\label{coef<} 0\le q^{(1)}(n,i),\, q^{(2)}(m,i)\le q(i) \text{ for
all } i \text{ and }\  q^{(2)}(m,m)=q(m).
\end{equation}
\end{corollary}

Denote
\begin{equation}
\label{genFunct} \varphi(z):=\Phi_{v,v|v}(G,\mathcal
W,z)=\sum_{i=1}^\infty q(i)z^i.
\end{equation}
It is self-evident that $\varphi_n$, $\varphi_{n,m}$ $\varphi$, and
all coefficients depend on $v$, but we somewhere omit $v$ from the
notation, because it is fixed.

It is clear that each of the functions $\varphi_n$, $\varphi_{n,m}$,
and $\varphi$ increases on the part of the semi-axis $\mathbb R^+$
where it is defined, and that the first two of them are polynomials
and hence are defined for all $z\ge 0$.

Let $R$ be the radius of convergence of the power series
\eqref{genFunct}. Since $(G,W)$ is a UPLG, we have
\begin{equation}
\label{unstab} 0<R<\infty, \ \ \varphi(R)=1,\ \ \varphi'(R)<\infty.
\end{equation}
We determine the numbers $R_n$ and $R_{n,m}$ by the equations
\begin{equation}
\label{defR} \varphi_n(R_n)= \varphi_{n,m}(R_{n,m})=1.
\end{equation}
Clearly $R<R_{n,m}<R_n$.

\begin{lemma}
\label{irregular} For every $n$
\begin{equation}
\label{limRnm}\varliminf_{m\to\infty}R_{n,m}=R.
\end{equation}
\end{lemma}
\begin{proof}
If the statement is not true, then
$\varliminf_{m\to\infty}R_{n,m}>R$ for some $n=\tilde n\in\mathbb N$
(because $R_{n,m}>R$ for all $n$ and $m>n$). Hence there are $\tilde
m$ and $\delta>0$ such that $R_{\tilde n,m}>R+\delta$ for all
$m>\tilde m$.

Since, by the Cauchy-Hadamard formula,
$\varlimsup_{m\to\infty}(q(m))^{1/m}=1/R$, one can find a sequence
of positive integers $m_i$ such that for all $i$
$$
(q_v(m_i))^{1/m_i}=R^{-1}+\alpha_i,
$$
where $\alpha_i\to 0$ as $i\to\infty$. Then
$q(m_i)=(R^{-1}+\alpha_i)^{m_i}$. Now, making use of the assumption
that $R_{\tilde n,m}>R+\delta$ for all $m>\tilde m$, one obtains,
for sufficiently large $i$,
$$
q(m_i)(R_{\tilde
n,m_i})^{m_i}\ge(R^{-1}+\alpha_i)^{m_i}(R+\delta)^{m_i}=(1+
\alpha_iR+\delta R^{-1}+\alpha_i\delta)^{m_i}.
$$
It is obvious that the last expression tends to infinity as
$i\to\infty$. Thus the same is true for $q(m_i)(R_{\tilde
n,m_i})^{m_i}$. But the coefficients of the polynomial
\eqref{genFunct2} are nonnegative and the coefficient of $z^m$
equals $q(m)$ (see \eqref{coef<}). Hence
$\lim_{i\to\infty}\varphi_{\tilde n,m_i}(R_{\tilde n,m_i})=\infty$,
which contradicts the definition of $R_{n,m_i}$ (see \eqref{defR}).
\end{proof}
\begin{lemma}
\label{seqPHI'} For every $n\in\mathbb N$ there exists a sequence of
positive integers $m_k>n$ such that
$$
\lim_{k\to\infty}m_k=\infty, \ \
\lim_{k\to\infty}\varphi'_{n,m_k}(R_{n,m_k})=\infty.
$$
\end{lemma}
\begin{proof}
Owing to \eqref{genFunct2}
\begin{align}
\label{otsenkaPHI'} \varphi'&_{n,m}(R_{n,m})=\varphi'_n(R_{n,m})+
\sum_{i=a(m)}^{k_m}iq^{(2)}(m,i)(R_{n,m})^{\,i-1}\notag\\
\ge\,& a(m)(R_{n,m})^{-1}
\sum_{i=a(m)}^{k_m}q^{(2)}(m,i)(R_{n,m})^{\,i}
=a(m)(R_{n,m})^{-1}\big[\varphi_{n,m}(R_{n,m})-
\varphi_n(R_{n,m})\big]\notag\\
=\,& a(m)(R_{n,m})^{-1}\big[1-\varphi_n(R_{n,m}) \big].
\end{align}
By Lemma \ref{seqPHI'} there is a sequence of positive integers
$m_k=m_k(n)$ such that $\lim_{k\to\infty}m_k=\infty$ and
$\lim_{{k\to\infty}}R_{n,m_k}=R$. Taking $m=m_k$ in
\eqref{otsenkaPHI'} and passing to the limit as $k\to\infty$, we
obtain $\lim_{k\to\infty}a(m_k)R^{-1}[1-\varphi_n(R)]$ on the
right-hand side. Because in this expression all factors independent
of $k$ are positive, while $a(m_k)\to\infty$ as $k\to\infty$, we
come to our statement.
\end{proof}

\begin{theorem}
\label{Unstab} There exists an increasing sequence of numbers
$n_k\in\mathbb N$ such that the sequence of subgraphs
$G_{n_k,n_{k+1}}\subset G$ is irregular.
\end{theorem}

\begin{proof}
The desired sequence will be constructed by induction. As $n_1$ one
can take an arbitrary positive integer, e.g., $n_1=1$. By Lemma
\ref{seqPHI'} there exists $n_2>n_1$ such that
$\varphi'_{n_1,n_2}(R_{n_1,n_2})>1$. Assume that $n_1,\dots,n_k$ for
which $\varphi'_{n_i,n_{i+1}}(R_{n_i,n_{i+1}})>i$ when $1\le i\le
k-1$ are already found. Using the same Lemma \ref{seqPHI'}, one can
find $n_{k+1}>n_k$ such that
$\varphi'_{n_k,n_{k+1}}(R_{n_k,n_{k+1}})>k$. We see that for the
sequence obtained,
$$
\lim_{k\to\infty}\varphi_{n_k,n_{k+1}} '(R_{n_k,n_{k+1}})=\infty,
$$
which implies irregularity of the exhaustive sequence of subgraphs
$G_{n_k,n_{k+1}}$ (see Lemma \ref{critReg}).
\end{proof}

\section{Existence of a regular sequence}
\label{Regular}
\smallskip

In this section we use the notation of the two previous ones without
recalling this every time.

In Section \ref{Irregular} an irregular sequence of finite
sub-graphs of a linear graph $G$ was constructed, leaning upon the
finiteness condition of Section \ref{LinGraph}. A regular sequence
will be constructed under a somewhat stronger {\it bounded jumps}
condition: there exists a constant $K>0$ such that for every forward
jump, i.e., an edge of the form $(i,j)$, where $i<j$, the inequality
$j-i<K$ asserts. Without loss of generality $K$ will be taken as
being a positive integer.

\begin{lemma}
\label{G(n)} For all $n\ge1$,
\begin{equation}
\label{ineq} (a)\ R_n\ge R,\ \ (b)\
\varphi'_n(R_n)\ge\varphi'_n(R),\ \ (c)\
\varphi'_n(R)\le\varphi'_{n+1}(R),
\end{equation}
and
\begin{equation}
\label{eq} \lim_{t\to\infty}\varphi'_n(R)=\varphi'(R).
\end{equation}
\end{lemma}

\begin{proof}
Inequalities (a) -- (c) in \eqref{ineq} are very simple, so we prove
only \eqref{eq}.

By the definition of $\varphi$
$$
\varphi'(R)=\sum_{k=1}^\infty kq(k)R^{\,k-1}.
$$
Owing to the stable positivity, this series converges (true, if it
even diverged, \eqref{eq} would assert as well). Then for each
$\varepsilon>0$, one can find $k(\varepsilon)\in\mathbb N$ such that
$$
\sum_{k=1}^{k(\varepsilon)} kq(k)R^{\,k-1}>\varphi'(R)-\varepsilon.
$$
By the definition of $G_n$ all simple $v$-cycles of length $\le n$
in $G$ are also simple $v$-cycles in $G_n$. Hence for all $n\ge
k(\varepsilon)$ we have $\varphi'_n(R)\ge\varphi'(R)-\varepsilon$.
Consequently, $\varliminf_{n\to\infty}\varphi'_n(R)\ge\varphi'(R)$.
On the other hand, it is clear that $\varphi'_n(R)\le\varphi'(R)$
for all $n$ (fc. (c) in \eqref{ineq}) and thus
$\varlimsup_{n\to\infty}\varphi'_n(R)\le\varphi'(R)$. From this we
obtain \eqref{eq}.
\end{proof}

\begin{lemma}
\label{deg} Let $P$ be a polynomial of degree ${\rm deg}\,P$ with
nonnegative coefficients and $d\ge{\rm deg}\,P$. Then for each pair
$a,b>0$ such that $a\le b$, we have $a^dP(b)\le b^{\,d}P(a)$.
\end{lemma}
\begin{proof}
Let $P(x):=\sum_{0\le i\le {\rm deg}\,P}c_ix^i$. Summing over all
$i$ the evident inequalities
$$
a^dc_ib^i\le b^dc_ia^i,\ \ 0\le i\le {\rm deg}\,P,
$$
we come to the statement of the lemma.
\end{proof}

\begin{theorem}
\label{regular} The sequence of the graphs $G_n$ (see
\eqref{sub-graphs}) is regular.
\end{theorem}
\begin{proof}
Owing to the bounded jumps condition (see \S\,\ref{Notat}), for each
simple $v$-cycle $\gamma$ in $G_n$ we have $M(\gamma)\le Kn$. Hence
$V(G_n)\subset\{1,2,\dots,Kn\}$ and $\ell(\gamma)\le Kn$. This
implies that the degree of the polynomial $\varphi'_n$ does not
exceed $Kn-1$.

Apply Lemma \ref{deg} to the polynomial $P(z):=\varphi'_n(z)$,
taking $d:=Kn$, $a:=R$, and $b:=R_n$ ($a\le b$ by inequality (a) in
\eqref{ineq}). By this lemma $R^{\,Kn}\varphi'_n(R_n)\le
(R_n)^{Kn}\varphi'_n(R)$, i.e.,
\begin{equation}
\label{phi'n}
\frac{\varphi'_n(R_n)}{\varphi'_n(R)}\le\frac{(R_n)^{Kn}}{R^{\,Kn}}=
\left(1+\frac{\delta_n}{R}\right)^{Kn},
\end{equation}
where $\delta_n=R_n-R$.

Let us estimate $\delta_n$. From the obvious convexity of the
function $\varphi_n$ we obtain
$$
\varphi_n(R_n)-\varphi_n(R)\ge\delta_n\varphi'_n(R),\ \text{ or }
\delta_n\le\frac{\varphi_n(R_n)-\varphi_n(R)}{\varphi'_n(R)}=
\frac{\varphi(R)-\varphi_n(R)}{\varphi'_n(R)}.
$$
When $n\to\infty$, the denominator of the last fraction tends to
$\varphi'(R)$ (see \eqref{eq}). Hence for all sufficiently big $n$,
\begin{equation}
\label{delta<} \delta_n\le C[\varphi(R)-\varphi_n(R)],
\end{equation}
where $1/2\varphi'(R)$ can be taken for $C$ .

Let us recall that the length of every simple $v$-cycle in the graph
$G_n$ does not exceed $Kn$. From this it follows that
$$
\varphi_n(z)=\sum_{i=1}^nq(i)z^i+\sum_{i=n+1}^{Kn}q(n,i)z^i,\ \ 0\le
q(n,i)\le q(i),
$$
and hence
\begin{align}
\label{nPhi<} n[\varphi(R)-\varphi_n(R)]= &
n\sum_{i=n+1}^{Kn}(q(i)-q(n,i))R^i+
n\sum_{i=Kn+1}^\infty q(i)R^i\notag \\
\le & n\sum_{i=n+1}^\infty q(i)R^i\le\sum_{i=n+1}^\infty iq(i)R^i.
\end{align}
The last sum is a remainder of the converging series representing
$\varphi'(R)$ and hence goes to zero as $n\to\infty$. Thus
$\delta_n=o(1/n)$ (see \eqref{delta<}). Together with \eqref{phi'n}
this yields
$$
\varlimsup_{n\to\infty}\frac{\varphi'_n(R_n)}{\varphi'_n(R)}\le 1,
$$
and since the denominator in the last fraction tends to
$\varphi'(R)$ (see \eqref{eq}), we obtain
$$
\varlimsup_{n\to\infty}\varphi'_n(R_n)\le\varphi'(R).
$$
On the other hand, it is evident that
$$
\varliminf_{n\to\infty}\varphi'_n(R_n)\ge
\varliminf_{n\to\infty}\varphi'_n(R)=\lim_{n\to\infty}\varphi'_n(R)=
\varphi'(R).
$$
Hence $\lim_{n\to\infty}\varphi'_n(R_n)=\varphi'(R)$, which proves
the theorem (see Lemma \ref{critReg}).
\end{proof}

\section{Concluding remarks}

1. The question on asymptotic behavior of equilibrium measures
corresponding to finite sub-graphs can be set up in a somewhat
different way. To do this we have to make use of the notion of a
principal subgraph of a directed graph. It is natural to introduce
this notion similarly to the notion of a principal sub-matrix of a
finite or infinite matrix (see \cite{HJ}).

We refer to a subgraph $\tilde G\subset G$ as {\it principal} (or
{\it induced}) if the set $E(\tilde G)$ of its edges consists of all
edges of $G$ whose initial and terminal vertices belong to $V(\tilde
G)$.

It turns out that if we require in the conjecture on regular and
irregular sequences that they consist of principal subgraphs, then
for irregular sequences it is false, while for regular ones the
question remains open. (Observe that in all cases known up to now
the conjecture was confirmed in this strengthened form.)

2. We now turn to irregular sequences. The corresponding sequences
of equilibrium measures converge to the null measure. By "mixing"
such a sequence with a regular one, we obtain a sequence of
subgraphs whose equilibrium measures have no limit. It is reasonable
to call such sequences irregular as well, thus generalizing this
notion. If such a sequence exists for every UPLG is unknown.

The author is greatful to S.A. Komech, V.I. Oseledets and S.A.
Pirogov for helpful comments.

\end{document}